\documentclass{article}
\usepackage{graphicx}
\usepackage[latin5]{inputenc}
\usepackage{amsmath,amssymb}
\usepackage{amsthm}
\usepackage{indentfirst}
\usepackage{graphicx}
\usepackage{pst-plot}
\usepackage{pstricks,pst-math,pst-infixplot}
\usepackage{pstricks-add}
\usepackage[dotinlabels]{titletoc}
\usepackage{pst-3dplot}
\usepackage{cite}
\usepackage{subfigure}
\usepackage[affil-it]{authblk}

\newcommand{\R}{\mathbb R}

\newtheorem{theorem}[section]{Theorem}
\newtheorem{corollary}[section]{Corollary}

\newtheorem{definition}[section]{Definition}

\newtheorem{example}[section]{Example}

\titlecontents{section}[0pt]{\addvspace{10pt}\bfseries}
  {\contentspush{\bfseries \thecontentslabel{.}\ \bfseries }}
  {\bfseries}{\titlerule*[8pt]{.}\contentspage}
\begin{document}
\title{A Vectorization for Nonconvex Set-valued Optimization \thanks{%
Mathematics Subject Classifications (2000): 80M50, 90C26.}
}

\author{Emrah Karaman$^1$, \.{I}lknur Atasever G\"{u}ven\c{c}$^1$, Mustafa Soyertem$^2$ \and Didem Tozkan$^1$, Mahide K\"{u}\c{c}\"{u}k$^1$, Yal\c{c}\i n K\"{u}\c{c}\"{u}k$^1$
}
\author[1]{Emrah Karaman$^1$, \.{I}lknur Atasever G\"{u}ven\c{c}$^1$, Mustafa Soyertem$^2$\thanks{corresponding author: mustafa.soyertem@usak.edu.tr} \and Didem Tozkan$^1$, Mahide K\"{u}\c{c}\"{u}k$^1$, Yal\c{c}\i n K\"{u}\c{c}\"{u}k$^1$}

\maketitle
$^1$
              Anadolu University, Yunus Emre Campus, Faculty of Science, Department of Mathematics, 26470, Eski\c{s}ehir, Turkey \and

              $^2$
              U\c{s}ak University, Bir Eyl\"{u}l Campus, Faculty of Art and Science and Department of Mathematics, 64200, U\c{s}ak, Turkey

\begin{abstract}
Vectorization is a technique that replaces a set-valued optimization problem with a vector optimization problem. In this work, by using an extension of Gerstewitz function \cite{Tammer}, a vectorizing function is defined to replace a given set-valued optimization problem with respect to set less order relation. Some properties of this function are studied. Also, relationships between a set-valued optimization problem and a vector optimization problem, derived via vectorization of this set-valued optimization problem, are examined. Furthermore, necessary and sufficient optimality conditions are presented without any convexity assumption.

\end{abstract}


Optimality conditions

\section*{Introduction}
Set-valued optimization, a generalization of vector optimization, has become a popular subject. Because it has many applications in game theory, engineering, control theory, finance, etc. \cite{Aubin1, Chinc, Hamel, Khan, Thompson, Neukel, Polak}. There are three types of solution concepts in set-valued optimization problems: based on vector approach \cite{Bot, Jahn3, Marin, Jahn, Khan, Ehrgott, Soyertem3, Atasever, Luc, Tanino}, based on set optimization approach \cite{Marin, Jahn2, Khan,  Kuroiwa, Kuroiwa1} and based on lattice structure \cite{Khan, Lohne}. In this work, we consider set optimization approach.

Kuroiwa et al. \cite{Kuroiwa1} presented six order relations for sets. Then, set optimization approach was introduced by Kuroiwa \cite{Kuroiwa3}. Later, Jahn and Ha defined new order relations and examined some properties of them \cite{Jahn1}.

There are some tools for solving set-valued optimization problems with respect to set optimization approach. Scalarization is one of them. Recently, some scalarization techniques obtained via Gerstewitz function have been widely used \cite{Molho, Marin, Khan}.

Hern\'{a}ndez and Rodr\'{\i}guez-Mar\'{\i}n examined relationships between solution concepts for vector approach and set optimization approach with respect to lower set less order relation. Moreover, they defined an extension of Gerstewitz function, obtained a nonconvex scalarization and optimality conditions for set-valued optimization problems with respect to lower set less order relation \cite{Marin}. E. Köbis and M. A. Köbis obtained nonconvex scalarizations with respect to several well-known set order relations \cite{Kobis}. Xu and Li presented a scalarization via oriented distance function and obtained optimality conditions for set-valued optimization problems with respect to upper set less order relation \cite{Li3}.

Vectorization is another tool for solving set-valued optimization problems by using vector-valued functions. This method replaces a set-valued optimization problem with a vector optimization problem which can be solved by using known methods such as numerical methods, scalarization etc. \cite{Bot, Ehrgott, Jahn, Luc}. Solutions obtained via these methods are also solutions of the set-valued optimization problem.

Vectorization based on total ordering cones was first introduced by K\"{u}\c{c}\"{u}k et al. \cite{Soyertem, Soyertem2}. They showed that a set-valued optimization problem can be represented as a vector-valued problem. They defined vectorizing function via existence and uniqueness of a minimal element of cone-closed and cone-bounded sets with respect to a total ordering cone. The value of this function at a point is the minimal element of the value of the set-valued map at this point with respect to the total ordering cone.

Another vectorization technique was given by Jahn \cite{Jahn2}. Jahn used linear approximations to define vectorizing function for set-valued optimization problem with respect to set less order relation \cite{Jahn2}. Under some convexity assumptions he gave optimality conditions for set-valued optimization problems with respect to set less order relation.

In the present study, a vectorizing function named Gerstewitz vectorizing function is defined by using an extension of Gerstewitz function to replace a set-valued optimization problem with respect to set less order relation with a vector optimization problem. A nonconvex approach to set-valued maps by using Gerstewitz vectorization is given. Moreover, necessary and sufficient optimality conditions for minimal and weak minimal solution are presented without any convexity assumption. Also, some examples for convex and nonconvex cases are used to demonstrate the usage of Gersewitz vectorization.

The paper is organized as follows: In section 2, we recall basic concepts of the theory of vector optimization and set-valued optimization. In section 3, we introduce Gerstewitz vectorizing function and examine some properties of this function. Moreover, relationships between this function and set less order relation are studied. In the last section, some optimality conditions are presented via Gerstewitz vectorization.

\section*{Preliminaries}
Throughout this paper, $X$ is any nonempty set, $Y$ denotes a real topological linear space ordered by a convex, closed and pointed cone $C\subset Y$ with nonempty interior. $\mathcal{P}_0(Y)$ is the notation of the family of all nonempty subsets of $Y$. Given any set $A\in\mathcal{P}_0(Y)$, $int(A)$ and $cl(A)$ are topological interior and the closure of  $A$, respectively. $\R^2$ is partially ordered by cone $\R^2_+$.

It is known that the cone $C$ induces the following ordering relations on $Y$ for $y,y'\in Y$ $$\begin{array}{ll}
 y\leq_C y' \Longleftrightarrow &  y'-y\in C\\
 y<_C y' \Longleftrightarrow & y'-y\in int(C).
 \end{array}$$
 Let $A\subset Y$ and $a_0\in A$. $a_0$ is a minimal (maximal) point of $A$ with respect to cone  $C$ if $A\cap(a_0-C)=\{a_0\}$ ($A\cap(a_0+C)=\{a_0\}$). The set of all minimal (maximal) points of $A$ is denoted by $\min A$ ($\max A$). Similarly, $a_0$ is a weak minimal (weak maximal) point of $A$ with respect to cone $C$  if $A\cap(a_0-int(C))=\emptyset$ ($A\cap(a_0+int(C))=\emptyset$) and the set of all weak minimal (weak maximal) points of $A$ is denoted by $W\min A$ ($W\max A$).

A vector optimization problem is defined by
   \begin{displaymath}
(VOP) \left\{ \begin{array}{ll}
\min(\max) f(x) & \\
 s.t. \ x\in X &
\end{array} \right.
\end{displaymath}
where $f:X\rightarrow Y$ is a vector valued function.
\begin{definition} \cite{Jahn}
  An element $\bar{x}\in X$ is called a minimal (maximal) solution of $(VOP)$ with respect to cone $C$ iff there isn't any $x\in X$ such that  $$f(x)\leq_Cf(\bar{x})\ \  (f(\bar{x})\leq_Cf(x)) \text{ and } f(x)\neq f(\bar{x}).$$
\end{definition}
\begin{definition}\cite{Jahn}
An element $\bar{x}\in X$ is called a minimal (maximal) strongly solution of $(VOP)$ with respect to cone $C$ iff  $$f(\bar{x})\leq_Cf(x) \ \  (f(x)\leq_Cf(\bar{x})) \ \ \ \text{for all } x\in X.$$
\end{definition}
If $\bar{x}\in X$ is a strongly solution of  $(VOP)$, it is also a solution of $(VOP)$ \cite{Jahn}.

It is said that $A$ is $C$-closed iff $A+C$ is a closed set; $C$-bounded iff for each neighborhood $U$ of zero in $Y$, there exists a positive real number $t$ such that $A\subset tU+C$; $C$-compact iff any cover of the form $\{U_\alpha+C \ | \ U_\alpha \text{ are open }, \alpha\in I\}$ admits a finite subcover. Every $C$-compact set is $C$-closed and $C$-bounded \cite{Luc}.

$A$ is called $\mp C$-bounded if $A\in\mathcal{P}_0(Y)$ is  $C$-bounded and $-C$-bounded in $Y$; if $A$ is $C$-closed and $-C$-closed, $A$ is called $\mp C$-closed; if $A$ is $C$-compact and $-C$-compact, $A$ is called $\mp C$-compact set. A set $A\in\mathcal{P}_0(Y)$ is called $C$-proper iff $A+C\neq Y$ and we denote  by $\mathcal{P}_{0C}(Y)$ the family of all $C$-proper subsets of $Y$ \cite{Marin}. A set $A\in\mathcal{P}_0(Y)$ is called $-C$-proper iff $A-C\neq Y$ and we denote by $\mathcal{P}_{0-C}(Y)$ the family of all $-C$-proper subsets of $Y$ \cite{Li3}. $\mathcal{P}^0_{\mp C}(Y)$ denotes the family of $C$-proper and $-C$-proper subsets of $Y$, namely, $\mathcal{P}^0_{\mp C}(Y):=\mathcal{P}_{0C}(Y)\cap\mathcal{P}_{0-C}(Y)$.

 Let $F:X\rightrightarrows Y$ be a set-valued map and $``N"$ denotes some property of a set in $Y$. $F$ is called $N$ valued on $X$ if $F(x)$ has the property $``N"$ for every $x\in X$. For example, if $F(x)$ is closed for all $x\in X$, we say that $F$ is closed valued on $X$.

 Let $F:X\rightrightarrows Y$ be a set-valued map and $F(x)\neq \emptyset$ for all $x\in X$. Set-valued optimization problem is defined by
\begin{displaymath}
(SOP) \left\{ \begin{array}{ll}
\min(\max) F(x) & \\
 s.t. \ x\in X. &
\end{array} \right.
\end{displaymath}

According to vector approach, we are looking for efficient points of the set \linebreak $\displaystyle F(X)=\bigcup_{x\in X}F(x)$ to solve $(SOP)$, that is, $x_0\in X$ is a solution of set-valued optimization problem if $$F(x_0)\cap\min\bigcup_{x\in X}F(x)\neq\emptyset \ \ \Big(F(x_0)\cap\max\bigcup_{x\in X}F(x)\neq\emptyset\Big).$$ When $(SOP)$ is considered according to vector approach, we denote the problem by $(v-SOP)$. Similarly, $x_0\in X$ is a weak solution of $(v-SOP)$ if $$F(x_0)\cap W\min\bigcup_{x\in X}F(x)\neq\emptyset \ \ \Big(F(x_0)\cap W\max\bigcup_{x\in X}F(x)\neq\emptyset\Big).$$

Set optimization approach is based on a comparison among the values of set-valued map \cite{Kuroiwa3}. That is, we are looking for efficient sets of the family \linebreak $\mathcal{F}(X)=\{F(x) \ | \ x\in X\}$ to solve $(SOP)$.

\begin{definition}\cite{Marin, Jahn1, Kuroiwa1, Li3} Let $A,B\in\mathcal{P}_0(Y)$.
\begin{itemize}
\item[(i)] lower set less order relation ($\preceq^{\ell}$) is defined by $A\preceq^{\ell} B\Longleftrightarrow B\subset A+C$,
\item[(ii)] strict lower set less order relation ($\prec^{\ell}$) is defined by $A\prec^{\ell} B\Longleftrightarrow B\subset A+int(C)$,
\item[(iii)] upper set less order relation ($\preceq^u$) is defined by $A\preceq^u B\Longleftrightarrow A\subset B-C$,
  \item[(iv)] strict upper set less order relation ($\prec^u$) is defined by $A\prec^u B\Longleftrightarrow A\subset B-int(C)$,
  \item[(v)] set less order relation ($\preceq^s$) is defined by $A\preceq^s B\Longleftrightarrow A\preceq^{\ell}B \text{ and } A\preceq^uB$,
\item[(vi)] strict set less order relation ($\prec^s$) is defined by $A\prec^s B\Longleftrightarrow A\prec^{\ell}B$ and $A\prec^uB$.
\end{itemize}
\end{definition}
Note that $\preceq^{\ell}$, $\preceq^u$ and $\preceq^s$ order relations are reflexive and transitive on $\mathcal{P}_0(Y)$. There is a relationship between $\preceq^{\ell}$ and $\preceq^u$: Let $A,B\in \mathcal{P}_0(Y)$, we have \begin{equation}
\label{eq4}
A\preceq^{\ell}B\Longleftrightarrow -B\preceq^u-A.\end{equation}

 Let $\sharp\in\{\ell,u,s\}$. $\sim^{\sharp}$ relation defined by $$A\sim^{\sharp} B\Longleftrightarrow A\preceq^{\sharp} B \text{ and } \  B\preceq^{\sharp} A$$ is an equivalence relation on $\mathcal{P}_0(Y)$. $[A]^{\sharp}$ denotes the equivalence class of $A$ with respect to $\sim^\sharp$, where $A\in \mathcal{P}_0(Y)$ \cite{Marin,Jahn1}.

Note that \begin{equation}
\label{eq5}
A\in[B]^{\ell}\Longleftrightarrow -A\in[-B]^u\end{equation} where $A,B\in \mathcal{P}_0(Y)$.

Now, we recall minimal, maximal, weak minimal and weak maximal set of a family with respect to order relations $\preceq^{\ell}$, $\preceq^u$ and $\preceq^s$.

\begin{definition}\cite{Marin, Jahn2}
Let $\mathcal{S}\subset\mathcal{P}_0(Y)$, $A\in\mathcal{S}$ and $\sharp\in\{\ell,u,s\}$ be given.
  \begin{itemize}
    \item[(i)] $A$ is said to be a $\sharp$-minimal set of $\mathcal{S}$ iff for any $B\in\mathcal{S}$ such that $B\preceq^{\sharp} A$ implies $A\preceq^{\sharp}B$. The family of $\sharp$-minimal sets of $\mathcal{S}$ is denoted by $\sharp-\min \mathcal{S}$.
    \item[(ii)] $A$ is said to be a $\sharp$-maximal set of $\mathcal{S}$ iff for any $B\in\mathcal{S}$ such that $A\preceq^{\sharp}B$ implies $B\preceq^{\sharp}A$. The family of $\sharp$-maximal sets of $\mathcal{S}$ is denoted by $\sharp-\max \mathcal{S}$.
  \end{itemize}
\end{definition}
\begin{definition}\cite{Marin,Jahn2}
Let $\mathcal{S}\subset\mathcal{P}_0(Y)$, $A\in\mathcal{S}$ and $\sharp\in\{\ell,u,s\}$ be given.
  \begin{itemize}
    \item[(i)] $A$ is called a weak $\sharp$-minimal set of $\mathcal{S}$ iff for any $B\in\mathcal{S}$ such that $B\prec^{\sharp} A$ implies $A\prec^{\sharp}B$. The family of weak $\sharp$-minimal sets of $\mathcal{S}$ is denoted by $\sharp-W\min \mathcal{S}$.
    \item[(ii)] $A$ is called a weak $\sharp$-maximal set of $\mathcal{S}$ iff for any $B\in\mathcal{S}$ such that $A\prec^{\sharp}B$ implies $B\prec^{\sharp}A$. The family of weak $\sharp$-maximal sets of $\mathcal{S}$ is denoted by $\sharp-W\max \mathcal{S}$.
  \end{itemize}
\end{definition}
Let $\sharp\in\{\ell,u,s\}$ and $(SOP)$ be given. According to the set optimization approach, if $F(x_0)$ is a $\sharp$-minimal ($\sharp$-maximal) set of $\mathcal{F}(X)$, then $x_0$ is called a solution of $(SOP)$ with respect to $\preceq^{\sharp}$. When $(SOP)$ is considered with respect to $\preceq^{\sharp}$, we denote it by $(\sharp-SOP)$. Similarly, if $F(x_0)$ is a weak $\sharp$-minimal (weak $\sharp$-maximal) set of $\mathcal{F}(X)$, then $x_0$ is called a weak solution of $(\sharp-SOP)$.

Note that if $F$ is a vector-valued function, solution(s) of $(v-SOP)$ coincides with solution(s) of $(\sharp-SOP)$.

The following definition is related with monotonicity of a real valued function defined on $\mathcal{P}_0(Y)$.
\begin{definition}\cite{Marin, Li3}
 Let  $\sharp\in\{\ell,u\}$ and $\mathcal{S}\subset\mathcal{P}_0(Y)$. A function $T:\mathcal{P}_0(Y)\rightarrow\R$ is called
  \begin{itemize}
    \item[(i)] $\sharp$-decreasing ($\sharp$-increasing) on $\mathcal{S}$ if $A,B\in \mathcal{S}$ and $A\preceq^{\sharp}B$ implies $T(B)\leq T(A)$ $(T(A)\leq T(B))$,
    \item[(ii)] strictly $\sharp$-decreasing (strictly $\sharp$-increasing) on $\mathcal{S}$ if $A,B\in \mathcal{S}$ and $A\prec^{\sharp}B$ implies $T(B)< T(A)$ $(T(A)< T(B))$.
  \end{itemize}
\end{definition}
Hern\'{a}ndez and Rodr\'{\i}guez-Mar\'{\i}n generalized Gerstewitz function as
\begin{equation}
\label{eq3}
  G_e(A,B)=\sup_{b\in B}\{\phi_{e,A}(b)\}
\end{equation}
where $e\in -int(C)$ and $\phi_{e,A}(y)=\inf\{t\in\R \ | \ y\in te+A+C\}$ \cite{Marin} and examined some properties of this function and obtained scalarization and optimality conditions for $(\ell-SOP)$. Throughout this paper, we use notation $G_e^{\ell}(\cdot,\cdot)$ instead of $G_e(\cdot,\cdot)$.

If we consider the nonconvex scalarization function $\phi^u_{e,A}(y)=\sup\{t\in\R \ | \ y\in te+A-C\}$, then one can obtain optimality conditions for $(u-SOP)$ similar to the conditions given by Hern\'{a}ndez and Rodr\'{\i}guez-Mar\'{\i}n in \cite{Marin}. In this function taking $A=\{0\}$ and $k=-e$ the equality $\phi^u_{-k,\{0\}}(y)=-z^{C,k}(y)$ is obtained, where $z^{C,k}$ is used to present nonconvex scalarization and  some optimality conditions with respect to $\preceq^u$, $\preceq^{\ell}$, $\preceq^s$ and $\preceq^{cert}$ by E. Köbis and M. A. Köbis in \cite{Kobis}.

\section*{Gerstewitz Vectorizing Function}

In this section, a vectorizing function is defined to replace a $(s-SOP)$ with $(VOP)$ using the generalized Gerstewitz function (\ref{eq3}). Some properties including monotonicity of this function are studied. Furthermore, relationships between this function and set less order relation are examined.

Now we give definition of monotonicity of a function from $\mathcal{P}^0_{\mp C}(Y)$ to $\R^2$.
\begin{definition}
   Let $\mathcal{A}\subset\mathcal{P}^0_{\mp C}(Y)$. A function $T:\mathcal{P}^0_{\mp C}(Y)\rightarrow\R^2$ is called
   \begin{itemize}
     \item[(i)] $s$-increasing ($s$-decreasing) on $\mathcal{A}$ if $A,B\in\mathcal{A}$ and $A\preceq^sB$ implies \linebreak $T(A)\leq_{\R^2_+}T(B) \ (T(B)\leq_{\R^2_+}T(A))$,
     \item[(ii)] strictly $s$-increasing (strictly $s$-decreasing) on $\mathcal{A}$ if $A,B\in\mathcal{A}$ and $A\prec^sB$ implies $T(A)<_{\R^2_+}T(B) \ (T(B)<_{\R^2_+}T(A))$.
   \end{itemize}
\end{definition}

Now, we introduce a vectorizing function which is the main tool to present a new vectorization.
\begin{definition}
 Let $A,B\in\mathcal{P}^0_{\mp C}(Y)$ and $e\in -int(C)$. The vectorizing function \linebreak $w_e:\mathcal{P}^0_{\mp C}(Y)\times\mathcal{P}^0_{\mp C}(Y)\rightarrow\overline{\R}^2$ defined by
 \begin{equation}
 \label{eqw}
   w_e(A,B)=\left(-G_e^{\ell}(A,B), -G_e^{\ell}(-B,-A)\right)
 \end{equation} is called Gerstewitz vectorizing function.
\end{definition}

Throughout this paper, in order to emphasize that the scalarization is adapted for $\preceq^u$ we simply use the notation $G_e^u(B,A)$ instead of $-G_e^{\ell}(-B,-A)$ where $A,B\in\mathcal{P}^0_{\mp C}(Y)$. Then, $$w_e(A,B)=\left(-G_e^{\ell}(A,B),G_e^u(B,A)\right)$$ for all $A,B\in\mathcal{P}^0_{\mp C}(Y)$.

 Here, some properties of $w_e(\cdot,\cdot)$ are stated.
\begin{theorem}
\label{u1}
 Let $A,B\in\mathcal{P}^0_{\mp C}(Y)$. Then the following statements are true:
 \begin{itemize}
   \item[(i)] If $A,B$ are $\mp C$-bounded, then $w_e(A,B)\in\R^2$,
   \item[(ii)] If $A\in[B]^s$, then $w_e(A,\cdot)=w_e(B,\cdot)$ and $w_e(\cdot,A)=w_e(\cdot,B)$,
   \item[(iii)] If $A\in[B]^s$, then $w_e(A,B)=w_e(B,A)$,
   \item[(iv)] $w_e(\cdot,A)$ is $s$-decreasing on $\mathcal{P}^0_{\mp C}(Y)$,
   \item[(v)] $w_e(A,\cdot)$ is $s$-increasing on $\mathcal{P}^0_{\mp C}(Y)$.
 \end{itemize}
\end{theorem}
\begin{proof}
  \begin{itemize}
    \item[(i)] Since $A$ and $B$ are $C$-bounded and $-C$-bounded, we have \linebreak $-G_e^{\ell}(A,B)\in\R$ and $G_e^u(B,A)\in\R$ from Theorem 3.6 of \cite{Marin}. Therefore, we obtain $w_e(A,B)\in\R^2$.
    \item[(ii)] Since $A\in[B]^{\ell}$ and $A\in[B]^u$, we have $-G_e^{\ell}(A,\cdot)=-G_e^{\ell}(B,\cdot)$ and \linebreak $G_e^u(\cdot,A)=G_e^u(\cdot,B)$ from Theorem 3.8 (i) and (iii) of \cite{Marin}, respectively. Therefore, we obtain $w_e(A,\cdot)=w_e(B,\cdot)$. Similarly, we get $w_e(\cdot,A)=w_e(\cdot,B)$ by using Theorem 3.8 (i) and (iii) of \cite{Marin}.
    \item[(iii)] Since $A\in[B]^{\ell}$ and $A\in[B]^u$, we have $-G_e^{\ell}(A,B)=-G_e^{\ell}(B,A)$ and \linebreak $G_e^u(A,B)=G_e^u(B,A)$ from Theorem 3.8 (iv), respectively. Then, we obtain \linebreak $w_e(A,B)=w_e(B,A)$.
    \item[(iv)] Assume that $B,D\in\mathcal{P}_{0\mp C}(Y)$ and $B\preceq^sD$. Then, $B\preceq^{\ell}D$ and $B\preceq^uD$. We have $-G_e^{\ell}(D,A)\leq -G_e^{\ell}(B,A)$ and $G_e^u(A,D)\leq G_e^u(A,B)$ from Theorem 3.8 (v) and (ii) of \cite{Marin}, respectively. Then, we have $w_e(D,A)\leq_{\R^2_+}w_e(B,A)$. Therefore $w_e(\cdot,A)$ is $s$-decreasing on $\mathcal{P}^0_{\mp C}(Y)$.
    \item[(v)] Assume that $B,D\in\mathcal{P}_{0\mp C}(Y)$ and $B\preceq^sD$. Then, $B\preceq^{\ell}D$ and $B\preceq^uD$. We have $-G_e^{\ell}(A,B)\leq -G_e^{\ell}(A,D)$ and $G_e^u(B,A)\leq G_e^u(D,A)$ from Theorem 3.8 (ii) and (v) of \cite{Marin}, respectively. Then, we have $w_e(A,B)\leq_{\R^2_+}w_e(A,D)$. Therefore, $w_e(A,\cdot)$ is $s$-increasing on $\mathcal{P}^0_{\mp C}(Y)$.
  \end{itemize}
\end{proof}
\begin{theorem}
\label{u2}
  Let $A\in\mathcal{P}^0_{\mp C}(Y)$ be a $\mp C$-compact set. Then the following statements are true:
  \begin{itemize}
     \item[(i)] $w_e(\cdot,A)$ is strictly $s$-decreasing on the family of $\mp C$-compact sets,
   \item[(ii)] $w_e(A,\cdot)$ is strictly $s$-increasing on the family of $\mp C$-compact sets.
  \end{itemize}
\end{theorem}
\begin{proof}
  \begin{itemize}
     \item[(i)] Assume that $B,D\in\mathcal{P}^0_{\mp C}(Y)$ are $\mp C$-compact sets and $B\prec^sD$. Then, \linebreak $B\prec^{\ell}D$ and $B\prec^uD$. We have $-G_e^{\ell}(D,A)< -G_e^{\ell}(B,A)$ and \linebreak $G_e^u(A,D)< G_e^u(A,B)$ from Theorem 3.9 (ii) and (i) of \cite{Marin}, respectively. Hence, we obtain $w_e(D,A)<_{\R^2_+}w_e(B,A)$. Therefore, $w_e(\cdot,A)$ is strictly $s$-decreasing on the family of $\mp C$-compact sets.
   \item[(ii)] This statement can be proved similar to (i) by using Theorem 3.9 (i) and (ii) of \cite{Marin}.
  \end{itemize}
  \end{proof}
  Now we examine relationships between set less order relation and Gerstewitz vectorizing function.

  Under different assumptions a necessary and sufficient condition similar to (iii) of Theorem \ref{ut1} was given by means of $z^{C,k}$ in \cite{Kobis}. These results are similar because if $k=-e$, then $G_e^{\ell}(A,B)=\underset{b\in B}{\sup} \ \underset{a\in A}{\inf}z^{C,k}(a-b)$.
  \begin{theorem}
  \label{ut1}
    Let $A\in\mathcal{P}^0_{\mp C}(Y)$ be a $\mp C$-closed set. Then, the following statements are true:
    \begin{itemize}
      \item[(i)] $w_e(A,A)=(0,0)$,
      \item[(ii)] If $A\in[B]^s$, then $w_e(A,B)=w_e(B,A)=(0,0)$,
      \item[(iii)] $A\preceq^sB$ if and only if $(0,0)\leq_{\R^2_+}w_e(A,B)$.
    \end{itemize}
  \end{theorem}
\begin{proof}
   \begin{itemize}
      \item[(i)] From Theorem 3.10 (i) of \cite{Marin} we have $-G_e^{\ell}(A,A)=0$ and $G_e^u(A,A)=0$. So, we obtain $$w_e(A,A)=(-G_e^{\ell}(A,A),G_e^u(A,A))=(0,0).$$
      \item[(ii)] Since $A\in[B]^{\ell}$ and $A\in[B]^u$, we have $-G_e^{\ell}(A,B)=-G_e^{\ell}(B,A)=0$ and \linebreak $G_e^u(A,B)=G_e^u(B,A)=0$ from Theorem 3.10 (ii) of \cite{Marin}, respectively. Therefore, $$w_e(A,B)=w_e(B,A)=(-G_e^{\ell}(A,B),G_e^u(B,A))=(0,0).$$
      \item[(iii)] $(\Longrightarrow)$ Let $A\preceq^sB$. Then, $A\preceq^{\ell}B$ and $A\preceq^uB$. Since $A\preceq^{\ell}B$ and $A\preceq^uB$, we have $-G_e^{\ell}(A,B)\geq0$ and $G_e^u(B,A)\geq0$ from Theorem 3.10 (iii) of \cite{Marin}, respectively. Thus, we obtain $(0,0)\leq_{\R^2_+}w_e(A,B)$.

          $(\Longleftarrow)$ Let $(0,0)\leq_{\R^2_+}w_e(A,B)$. Then, we have $G_e^{\ell}(A,B)\leq0$ and $G_e^u(B,A)\geq0$. So, $A\preceq^{\ell}B$ and $A\preceq^uB$ from Theorem 3.10 (iii) of \cite{Marin}, respectively. Therefore, $A\preceq^sB$.
    \end{itemize}
\end{proof}
\begin{theorem}
\label{u7}
  Let $A,B\in\mathcal{P}^0_{\mp C}(Y)$ be $\mp C$-compact sets. Then, $$(0,0)<_{\R^2_+}w_e(A,B)\Longleftrightarrow A\prec^sB.$$

\end{theorem}
\begin{proof}
$(\Longrightarrow)$ Let $(0,0)<_{\R^2_+}w_e(A,B)$. Then, $G_e^{\ell}(A,B)<0$ and $G_e^u(B,A)>0$. Since $G_e^{\ell}(A,B)<0$ and $G_e^u(B,A)>0$, we have $A\prec^{\ell}B$ and $A\prec^uB$ from Corollary 3.11 (i) of \cite{Marin}, respectively. Hence, $A\prec^sB$.

          $(\Longleftarrow)$ Let $A\prec^sB$. Then, $A\prec^{\ell}B$ and $A\prec^uB$. As $A\prec^{\ell}B$ and $A\prec^uB$, we have $G_e^{\ell}(A,B)<0$ and $G_e^u(B,A)>0$ from Corollary 3.11 (i) of \cite{Marin}, respectively. Therefore, we obtain $(0,0)<_{\R^2_+}w_e(A,B)$.
\end{proof}

 We define a vectorizing function $v_e:\mathcal{P}^0_{\mp C}(Y)\rightarrow{\overline{\R}}^2$ as

$$v_e(A):=w_e(\{0\},A)=\left(-G_e^{\ell}(\{0\},A),G_e^u(A,\{0\})\right)$$ where $e\in-int(C)$ in order to use the advantage of computation of a single variable function.

It can be seen that $v_e(\cdot)$ is $s$-increasing on $\mathcal{P}^0_{\mp C}(Y)$ and strictly $s$-increasing on the family of $\mp C$-compact sets. Let $A,B\in\mathcal{P}^0_{\mp C}(Y)$. If $A\in[B]^s$, then $v_e(A)=v_e(B)$.

 By taking [19, Example 3.1] we demonstrate the calculations of $w_e(\cdot,\cdot)$ and $v_e(\cdot)$.

\begin{example}
  Let $Y=\R^2$, $C=\R^2_+$ and $F:[-1,1]\rightrightarrows Y$ be defined as  $$F(x):=\left\{(y_1,y_2)\in\R^2 \ | \ (y_1-2x^2)^2+(y_2-2x^2)^2\leq(x^2+1)^2\right\}$$ for all $x\in[-1,1]$ (Fig. \ref{fig18}).

\begin{figure}[h]
\begin{center}
\psset{unit=.75cm}
\begin{pspicture}(-1.9,-1.9)(6.1,6.1)
\psaxes[arrows=->](0,0)(-1.9,-1.9)(6,6)
\rput[h](-0.2,-0.2){\black{{\textbf{O}}}}
\rput[h](6.2,0){\black{$x$}}
\rput[h](0,6.2){\black{$y$}}
\pscircle[linecolor=black,linewidth=1pt,fillstyle=vlines,hatchwidth=.01](0,0){1}
\rput[h](-1.3,.5){\footnotesize$F(0)$}
\pscircle[linecolor=blue,linewidth=1pt,fillstyle=hlines,hatchcolor=blue,hatchwidth=.01](.5,.5){1.25}
\rput[h](2,-1){\blue\footnotesize$F(\frac{1}{2})=F(-\frac{1}{2})$}
\pscircle[linecolor=red,linewidth=1pt,fillstyle=vlines,hatchcolor=red,hatchwidth=.01](2,2){2}
\rput[h](4,4){\red\footnotesize$F(1)=F(-1)$}
\end{pspicture}
\end{center}
\caption{Some image sets of $F$}
\label{fig18}
\end{figure}
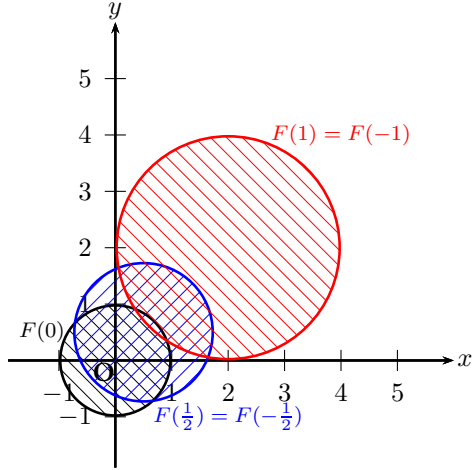

First, we choose $e=\left(-\frac{\sqrt{2}}{2},-\frac{\sqrt{2}}{2}\right)$ to find $w_e(F(x),F(0))$ and $v_e(F(x))$ for an arbitrary $x\in[-1,1]$. So, we have to calculate $G_e^{\ell}(F(x),F(0))$ and $G_e^u(F(0),F(x))$. To find the value of $$G_e^{\ell}(F(x),F(0))=\min\{t\in\R \ | \ F(0)\subset te+F(x)+C\}$$ we should evaluate the smallest $t$ that allows $te+F(x)+C$ to cover $F(0)$. This value means how long at least $F(x)+C$ should move along the direction $e$ to cover the set $F(0)$. We achieve this smallest value clearly by substracting difference of radii of $F(x)$ and $F(0)$ from the distance between centers of these balls as seen in Fig. \ref{fig19}. Hence we get $G_e^{\ell}(F(x),F(0))=2\sqrt{2}x^2-x^2$.

\begin{figure}[h]
\begin{center}
\psset{Dy=4,Dx=4,xunit=9mm,yunit=9mm}
\subfigure[$F(0)$, $F(x)$ and $e$]{\begin{pspicture}(-1,-1.5)(3.4,3.4)
\psaxes[arrows=->](0,0)(-1.2,-1.2)(3,3)
\rput[h](-0.2,0.2){\black{{\footnotesize\textbf{O}}}}
\rput[h](3.1,0){\black{$x$}}
\rput[h](0,3.1){\black{$y$}}
\pscircle[linecolor=black,linewidth=1pt](0,0){1}
\rput[h](-1,-1){\footnotesize$F(0)$}
\pscircle[linecolor=blue,linewidth=1pt](.9,.9){1.5}
\rput[h](1,2.8){\blue\footnotesize$F(x)$}
\psline[linecolor=red,linewidth=.5pt]{<-}(-.73,-.73)(0,0)
\rput[h](-.6,-.4){\red\small$e$}
\rput[h](1.2,1.2){\blue\tiny$(2x^2,2x^2)$}
\psdot[linewidth=.5pt,linecolor=blue](.9,.9)
\psline[linewidth=.7pt](.9,.9)(2.5,.9)
\rput[h](2,.7){\tiny$x^2+1$}
\end{pspicture}}
\subfigure[$F(0)$, $F(x)$ and $F(x)+C$]{\begin{pspicture}(-1,-1.5)(3.4,3.4)
\psaxes[arrows=->](0,0)(-1.2,-1.2)(3,3)
\rput[h](-0.2,0.2){\black{{\footnotesize\textbf{O}}}}
\rput[h](3.1,0){\black{$x$}}
\rput[h](0,3.1){\black{$y$}}
\pscircle[linecolor=black,linewidth=1pt](0,0){1}
\rput[h](-1,-1){\footnotesize$F(0)$}
\pscustom[linecolor=white,linewidth=0.001pt,fillstyle=hlines,ticks=none,hatchwidth=.5pt,hatchcolor=cyan]{
\pscircle[linecolor=blue,fillstyle=solid,hatchwidth=.5pt,fillcolor=white,hatchcolor=white](.9,.9){1.5}\pscurve(-.6,.9)(.6,-.9)}
\pscustom[linecolor=white,linewidth=0.001pt,fillstyle=hlines,ticks=none,hatchwidth=.5pt,hatchcolor=cyan]{
\psline(-.7,.4)(-.7,3)
\psline(-.7,3)(3,3)
\psline(3,3)(3,-.7)
\psline(3,-.7)(.4,-.7)}
\psline[linecolor=cyan,linewidth=1pt](-.75,.67)(-.75,3)
\psline[linecolor=cyan,linewidth=1pt](.67,-.75)(3,-.75)
\pscircle[linecolor=cyan,linewidth=1pt](.9,.9){1.5}
\rput[h](1,3.3){\cyan\footnotesize$F(x)+C$}
\psline[linecolor=red,linewidth=.5pt]{<->}(-.73,-.73)(-.25,-.25)
\rput[h](-.9,-.5){\red\tiny$2\sqrt{2}x^2-x^2$}
\pscircle[linecolor=blue,linewidth=1pt](.9,.9){1.5}
\rput[h](1,2.8){\blue\footnotesize$F(x)$}
\end{pspicture}}
\subfigure[$F(0)$, $F(x)$ and $(2\sqrt{2}x^2-x^2)e+F(x)+C$]{\begin{pspicture}(-1,-1.5)(3.4,3.4)
\psaxes[arrows=->](0,0)(-1.2,-1.2)(3,3)
\rput[h](-0.2,0.2){\black{{\footnotesize\textbf{O}}}}
\rput[h](3.1,0){\black{$x$}}
\rput[h](0,3.1){\black{$y$}}
\pscustom[linecolor=white,linewidth=0.001pt,fillstyle=hlines,ticks=none,hatchwidth=.5pt,hatchcolor=cyan]{
\pscircle[linecolor=blue,fillstyle=solid,hatchwidth=.5pt,fillcolor=white,hatchcolor=white](.4,.4){1.5}\pscurve(.1,.4)(.1,-.4)}
\pscustom[linecolor=white,linewidth=0.001pt,fillstyle=hlines,ticks=none,hatchwidth=.5pt,hatchcolor=cyan]{
\psline(-1.2,-.1)(-1.2,3)
\psline(-1.2,3)(3,3)
\psline(3,3)(3,-1.2)
\psline(3,-1.2)(-.1,-1.2)}
\psline[linecolor=cyan,linewidth=1pt](-1.25,.17)(-1.25,3)
\psline[linecolor=cyan,linewidth=1pt](.17,-1.25)(3,-1.25)
\pscircle[linecolor=cyan,linewidth=1pt](.4,.4){1.5}
\rput[h](2,3.3){\cyan\footnotesize$(2\sqrt{2}x^2-x^2)e+F(x)+C$}
\rput[h](-.5,-.4){\footnotesize$F(0)$}
\pscircle[linecolor=blue,linewidth=1pt](.9,.9){1.5}
\rput[h](1,2.8){\blue\footnotesize$F(x)$}
\pscircle[linecolor=black,linewidth=1pt](0,0){1}
\end{pspicture}}\end{center}
\caption{The calculation of $G_e^{\ell}(F(x),F(0))$ geometrically}
\label{fig19}
\end{figure}

With a similar manner, to find $$G_e^u(F(0),F(x))=\max\{t\in\R \ | \ F(x)\subset te+F(0)-C\}$$ we calculate the largest value of $t$ that allows $F(x)$ to be covered by $te+F(0)-C$. \linebreak This largest value is clearly negative of the distance between the vectors \linebreak $(3x^2+1,3x^2+1)$ and $(1,1)$ as seen in Fig. \ref{fig20}. Thus, we have  $$G_e^u(F(0),F(x))=-3\sqrt{2}x^2.$$

 Finally, we have
\begin{equation}
\label{eq1}
\begin{array}{ll}
                        w_e(F(x),F(0)) & =\left(-G_e^{\ell}(F(x),F(0)),G_e^u(F(0),F(x))\right) \\
                         & =\left((1-2\sqrt{2})x^2,-3\sqrt{2}x^2\right).
                      \end{array}\end{equation}

\begin{figure}[h]
\begin{center}
\psset{Dy=4,Dx=4,xunit=8.9mm,yunit=8.9mm}
\subfigure[$F(0)$, $F(x)$ and $e$]{\begin{pspicture}(-1,-2)(3.5,3.5)
\psaxes[arrows=->](0,0)(-1.2,-1.2)(3,3)
\rput[h](-0.2,0.2){\black{{\footnotesize\textbf{O}}}}
\rput[h](3.1,0){\black{$x$}}
\rput[h](0,3.1){\black{$y$}}
\pscircle[linecolor=black,linewidth=1pt](0,0){1}
\rput[h](-1,-1){\footnotesize$F(0)$}
\pscircle[linecolor=blue,linewidth=1pt](.9,.9){1.5}
\rput[h](1,2.8){\blue\footnotesize$F(x)$}
\psline[linecolor=red,linewidth=.5pt]{<-}(-.73,-.73)(0,0)
\rput[h](-.6,-.4){\red\small$e$}
\rput[h](1.2,1.2){\blue\tiny$(2x^2,2x^2)$}
\psdot[linewidth=.5pt,linecolor=blue](.9,.9)
\psline[linewidth=.7pt](.9,.9)(2.5,.9)
\rput[h](2,.7){\tiny$x^2+1$}
\end{pspicture}}
\subfigure[$F(0)$, $F(x)$ and $F(0)-C$]{\begin{pspicture}(-1,-2)(3.4,3.4)
\psaxes[arrows=->](0,0)(-1.2,-1.2)(3,3)
\rput[h](-0.2,0.2){\black{{\footnotesize\textbf{O}}}}
\rput[h](3.1,0){\black{$x$}}
\rput[h](0,3.1){\black{$y$}}
\pscircle[linecolor=cyan,linewidth=1pt](0,0){1}
\pscustom[linecolor=white,linewidth=0.001pt,fillstyle=hlines,ticks=none,hatchwidth=.5pt,hatchcolor=cyan]{
\pscircle[linecolor=blue,fillstyle=solid,hatchwidth=.5pt,fillcolor=white,hatchcolor=white](0,0){1}\pscurve(-1.1,0)(0,1.1)}
\pscustom[linecolor=white,linewidth=0.001pt,fillstyle=hlines,ticks=none,hatchwidth=.5pt,hatchcolor=cyan]{
\psline(-1.3,1.1)(0,1.1)
\psline(0,1.1)(0,-1.3)
\psline(0,-1.3)(-1.3,-1.3)
\psline(-1.3,-1.3)(-1.3,1.1)}
\pscustom[linecolor=white,linewidth=0.001pt,fillstyle=hlines,ticks=none,hatchwidth=.5pt,hatchcolor=cyan]{
\psline(0,1.1)(1.1,0)
\psline(1.1,0)(0,0)
\psline(0,0)(0,1.1)}
\rput[h](1,2.8){\blue\footnotesize$F(x)$}
\pscustom[linecolor=white,linewidth=0.001pt,fillstyle=hlines,ticks=none,hatchwidth=.5pt,hatchcolor=cyan]{
\psline(0,0)(1.1,0)
\psline(1.1,0)(1.1,-1.3)
\psline(1.1,-1.3)(0,-1.3)
\psline(0,-1.3)(0,0)}
\psline[linecolor=cyan,linewidth=1pt](1.1,0)(1.1,-1.3)
\psline[linecolor=cyan,linewidth=1pt](0,1.1)(-1.3,1.1)
\pscircle[linecolor=blue,linewidth=1pt](.9,.9){1.5}
\rput[h](-.5,-1.5){\cyan\footnotesize$F(0)-C$}
\psline[linecolor=red,linewidth=.5pt]{<->}(1.1,1.1)(2.58,2.58)
\psline[linestyle=dashed,linewidth=.5pt](1.1,0)(1.1,1.1)
\psline[linestyle=dashed,linewidth=.5pt](0,1.1)(1.1,1.1)
\psline[linestyle=dashed,linewidth=.5pt](2.58,0)(2.58,2.58)
\psline[linestyle=dashed,linewidth=.5pt](0,2.58)(2.58,2.58)
\rput[h](1.5,2){\red\tiny$3\sqrt{2}x^2$}
\pscircle[linecolor=black,linewidth=1pt](0,0){1}
\rput[h](-1,-1){\footnotesize$F(0)$}
\rput[h](1.5,1){\tiny$(1,1)$}
\psdot[linewidth=.5pt](1.1,1.1)
\psdot[linewidth=.5pt](2.56,2.56)
\psaxes[arrows=->](0,0)(-1.2,-1.2)(3,3)
\end{pspicture}}
\subfigure[$F(0)$, $F(x)$ and $3\sqrt{2}x^2e+F(0)-C$]{\begin{pspicture}(-1,-2)(3.4,3.4)
\psaxes[arrows=->](0,0)(-1.2,-1.2)(3,3)
\rput[h](-0.2,0.2){\black{{\footnotesize\textbf{O}}}}
\rput[h](3.1,0){\black{$x$}}
\rput[h](0,3.1){\black{$y$}}
\pscircle[linecolor=cyan,linewidth=1pt](1.5,1.5){1}
\pscustom[linecolor=white,linewidth=0.001pt,fillstyle=hlines,ticks=none,hatchwidth=.5pt,hatchcolor=cyan]{
\pscircle[linecolor=blue,fillstyle=solid,hatchwidth=.5pt,fillcolor=white,hatchcolor=white](1.5,1.5){1}\pscurve(1.1,0)(0,1.1)}
\pscustom[linecolor=white,linewidth=0.001pt,fillstyle=hlines,ticks=none,hatchwidth=.5pt,hatchcolor=cyan]{
\psline(-1.3,2.58)(1,2.58)
\psline(1,2.58)(1,-1.3)
\psline(1,-1.3)(-1.3,-1.3)
\psline(-1.3,-1.3)(-1.3,2.58)}
\pscustom[linecolor=white,linewidth=0.001pt,fillstyle=hlines,ticks=none,hatchwidth=.5pt,hatchcolor=cyan]{
\psline(0,1.1)(1.1,0)
\psline(1.1,0)(0,0)
\psline(0,0)(0,1.1)}
\pscustom[linecolor=white,linewidth=0.001pt,fillstyle=hlines,ticks=none,hatchwidth=.5pt,hatchcolor=cyan]{
\psline(2.58,-1.3)(0,-1.3)
\psline(0,-1.3)(0,1.58)
\psline(0,1.58)(2.58,1.58)
\psline(2.58,1.58)(2.58,-1.3)}
\psline[linecolor=cyan,linewidth=1pt](-1.3,2.58)(1.58,2.58)
\psline[linecolor=cyan,linewidth=1pt](2.58,-1.3)(2.58,1.58)
\pscircle[linecolor=blue,linewidth=1pt](.9,.9){1.5}
\rput[h](1,-1.5){\cyan\footnotesize$3\sqrt{2}x^2e+F(0)-C$}
\rput[h](1.4,-1){\blue\footnotesize$F(x)$}
\pscircle[linecolor=black,linewidth=1pt](0,0){1}
\rput[h](-1,-1){\footnotesize$F(0)$}
\psaxes[arrows=->](0,0)(-1.2,-1.2)(3,3)
\end{pspicture}}\end{center}

 \
\caption{The calculation of $G_e^u(F(0),F(x))$ geometrically}
\label{fig20}
\end{figure}
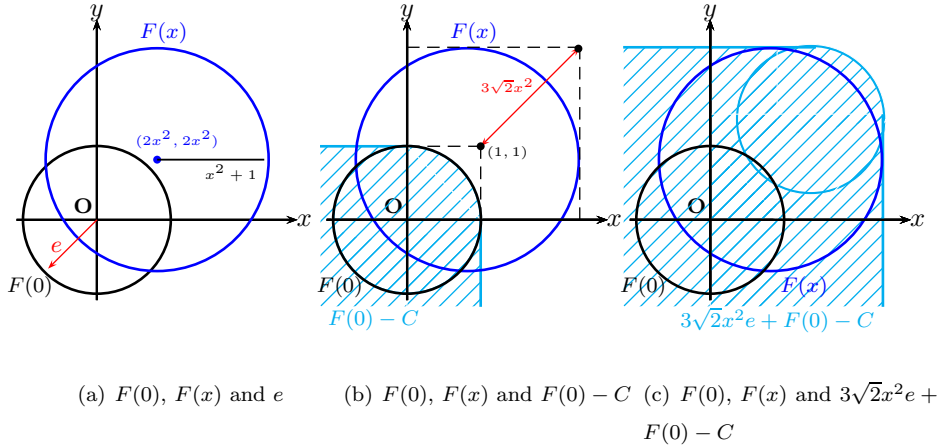

Now, we will find $v_e(F(x))$ for all $x\in[-1,1]$. We need to calculate $G_e^{\ell}(\{0\},F(x))$ and $G_e^u(F(x),\{0\})$ for all $x\in[-1,1]$. If $C$ is moved along the direction $e$ until it covers $F(x)$, then we get value of  $$G_e^{\ell}(\{0\},F(x))=\min\{t\in\R \ | \ F(x)\subset te+C\}=\sqrt{2}(1-x^2)$$ by substracting distance between center of $F(x)$ and $\left(-\frac{\sqrt{2}}{2},-\frac{\sqrt{2}}{2}\right)$, and radius of $F(x)$. As seen in Fig. \ref{fig21}, we get  $G_e^{\ell}(\{0\},F(x))=\sqrt{2}(1-x^2)$.

\begin{figure}[h]
\begin{center}
\psset{Dy=4,Dx=4,xunit=8.9mm,yunit=8.9mm}
\subfigure[$F(x)$, $C$ and $e$]{\begin{pspicture}(-1,-1.3)(3.4,3.4)
\psaxes[arrows=->](0,0)(-1,-1)(3,3)
\rput[h](-0.2,0.2){\black{{\footnotesize\textbf{O}}}}
\pscustom[linecolor=white,linewidth=0.001pt,fillstyle=hlines,ticks=none,hatchwidth=.5pt,hatchcolor=cyan]{
\psline(0,0)(2.9,0)
\psline(2.9,0)(2.9,2.9)
\psline(2.9,2.9)(0,2.9)
\psline(0,2.9)(0,0)}
\psline[linecolor=cyan,linewidth=1pt](0,0)(0,2.9)
\psline[linecolor=cyan,linewidth=1pt](0,0)(2.9,0)
\rput[h](3.1,0){\black{$x$}}
\rput[h](0,3.1){\black{$y$}}
\pscircle[linecolor=blue,linewidth=1pt](.5,.5){1.25}
\rput[h](1,2){\blue\footnotesize$F(x)$}
\rput[h](2.5,3.2){\cyan\footnotesize$C$}
\psline[linecolor=red,linewidth=.5pt]{<-}(-.73,-.73)(0,0)
\psline[linestyle=dashed,linewidth=.5pt](-.75,0)(-.75,-.75)
\psline[linestyle=dashed,linewidth=.5pt](0,-.75)(-.75,-.75)
\rput[h](-.98,.28){\tiny$-\frac{\sqrt{2}}{2}$}
\psdot[linewidth=.5pt](-.75,0)
\rput[h](.3,-.85){\tiny$-\frac{\sqrt{2}}{2}$}
\psdot[linewidth=.5pt](0,-.75)
\rput[h](-.4,-.2){\red\small$e$}
\rput[h](.7,.7){\blue\tiny$(2x^2,2x^2)$}
\psdot[linewidth=.5pt,linecolor=blue](.5,.5)
\psline[linewidth=.7pt](.5,.5)(1.85,.5)
\rput[h](1.3,.3){\tiny$x^2+1$}
\end{pspicture}}
\subfigure[$F(x)$]{\begin{pspicture}(-1,-1.3)(3.4,3.4)
\psaxes[arrows=->](0,0)(-1,-1)(3,3)
\rput[h](-0.2,0.2){\black{{\footnotesize\textbf{O}}}}
\rput[h](3.1,0){\black{$x$}}
\rput[h](0,3.1){\black{$y$}}
\rput[h](1,2){\blue\footnotesize$F(x)$}
\pscircle[linecolor=blue,linewidth=1pt](.5,.5){1.25}
\psline[linecolor=red,linewidth=.5pt]{<->}(0,0)(-.9,-.9)
\rput[h](-1.18,-.5){\red\tiny$\sqrt{2}(1-x^2)$}
\psline[linestyle=dashed,linewidth=.5pt](-.9,.5)(-.9,-.9)
\psline[linestyle=dashed,linewidth=.5pt](.5,-.9)(-.9,-.9)
\end{pspicture}}
\subfigure[$F(x)$ and $\sqrt{2}(1-x^2)e+C$]{\begin{pspicture}(-1,-1)(3.4,3.4)
\psaxes[arrows=->](0,0)(-1,-1)(3,3)
\rput[h](-0.2,0.2){\black{{\footnotesize\textbf{O}}}}
\rput[h](3.1,0){\black{$x$}}
\rput[h](0,3.1){\black{$y$}}
\pscustom[linecolor=white,linewidth=0.001pt,fillstyle=hlines,ticks=none,hatchwidth=.5pt,hatchcolor=cyan]{
\psline(-.9,-.9)(-.9,2.9)
\psline(-.9,2.9)(2.9,2.9)
\psline(2.9,2.9)(2.9,-.9)
\psline(2.9,-.9)(-.9,-.9)}
\psline[linecolor=cyan,linewidth=1pt](-.9,2.9)(-.9,-.9)
\psline[linecolor=cyan,linewidth=1pt](2.9,-.9)(-.9,-.9)
\pscircle[linecolor=blue,linewidth=1pt](.5,.5){1.25}
\rput[h](1.5,3.2){\cyan\footnotesize$\sqrt{2}(1-x^2)e+C$}
\rput[h](1,2){\blue\footnotesize$F(x)$}
\end{pspicture}}\end{center}
\caption{The calculation of $G_e^{\ell}(\{0\},F(x))$ geometrically}
\label{fig21}
\end{figure}

To calculate $G_e^u(F(x),\{0\})$, we can use the formula $$G_e^u(F(x),\{0\})=\max\{t\in\R \ | \ 0\in te+F(x)-C\}.$$ As seen in Fig. \ref{fig22}, this value can be found by adding distance between origin and center of $F(x)$, and radius of $F(x)$. Hence, we get  $G_e^u(F(x),\{0\})=2\sqrt{2}x^2+x^2+1$.
\begin{figure}[h]
\begin{center}
\psset{Dy=4,Dx=4,xunit=9mm,yunit=9mm}
\subfigure[$F(x)$, $C$ and $e$]{\begin{pspicture}(-1,-1.5)(3.4,3.4)
\psaxes[arrows=->](0,0)(-1,-1)(3,3)
\rput[h](-0.2,0.2){\black{{\footnotesize\textbf{O}}}}
\pscustom[linecolor=white,linewidth=0.001pt,fillstyle=hlines,ticks=none,hatchwidth=.5pt,hatchcolor=cyan]{
\psline(0,0)(2.9,0)
\psline(2.9,0)(2.9,2.9)
\psline(2.9,2.9)(0,2.9)
\psline(0,2.9)(0,0)}
\psline[linecolor=cyan,linewidth=1pt](0,0)(0,2.9)
\psline[linecolor=cyan,linewidth=1pt](0,0)(2.9,0)
\rput[h](3.1,0){\black{$x$}}
\rput[h](0,3.1){\black{$y$}}
\pscircle[linecolor=blue,linewidth=1pt](.5,.5){1.25}
\rput[h](1,2){\blue\footnotesize$F(x)$}
\rput[h](2.5,3.2){\cyan\footnotesize$C$}
\psline[linecolor=red,linewidth=.5pt]{<-}(-.73,-.73)(0,0)
\psline[linestyle=dashed,linewidth=.5pt](-.75,0)(-.75,-.75)
\psline[linestyle=dashed,linewidth=.5pt](0,-.75)(-.75,-.75)
\rput[h](-.98,.28){\tiny$-\frac{\sqrt{2}}{2}$}
\psdot[linewidth=.5pt](-.75,0)
\rput[h](.3,-.85){\tiny$-\frac{\sqrt{2}}{2}$}
\psdot[linewidth=.5pt](0,-.75)
\rput[h](-.4,-.2){\red\small$e$}
\rput[h](.7,.7){\blue\tiny$(2x^2,2x^2)$}
\psdot[linewidth=.5pt,linecolor=blue](.5,.5)
\psline[linewidth=.7pt](.5,.5)(1.85,.5)
\rput[h](1.3,.3){\tiny$x^2+1$}
\end{pspicture}}
\subfigure[$F(x)$ and $F(x)-C$]{\begin{pspicture}(-1,-1.5)(3.4,3.4)
\psaxes[arrows=->](0,0)(-1,-1)(3,3)
\rput[h](-0.2,0.2){\black{{\footnotesize\textbf{O}}}}
\rput[h](3.1,0){\black{$x$}}
\rput[h](0,3.1){\black{$y$}}
\rput[h](0,-1.2){\cyan\footnotesize$F(x)-C$}
\pscustom[linecolor=white,linewidth=0.001pt,fillstyle=hlines,ticks=none,hatchwidth=.5pt,hatchcolor=cyan]{
\pscircle[linecolor=blue,fillstyle=solid,hatchwidth=.5pt,fillcolor=white,hatchcolor=white](.5,.5){1.25}\pscurve(1.1,0)(0,1.1)}
\pscustom[linecolor=white,linewidth=0.001pt,fillstyle=hlines,ticks=none,hatchwidth=.5pt,hatchcolor=cyan]{
\psline(-1,1.8)(.2,1.8)
\psline(.2,1.8)(.2,-1)
\psline(.2,-1)(-1,-1)
\psline(-1,-1)(-1,1.8)}
\pscustom[linecolor=white,linewidth=0.001pt,fillstyle=hlines,ticks=none,hatchwidth=.5pt,hatchcolor=cyan]{
\psline(-.2,.2)(.2,-.2)
\psline(.2,-.4)(-.2,-.2)
\psline(-.2,-.2)(-.2,.2)}
\pscustom[linecolor=white,linewidth=0.001pt,fillstyle=hlines,ticks=none,hatchwidth=.5pt,hatchcolor=cyan]{
\psline(1.87,-1)(-.4,-1)
\psline(-.4,-1)(-.4,1)
\psline(-.4,1)(1.87,1)
\psline(1.87,1)(1.87,-1)}
\psline[linecolor=cyan,linewidth=1pt](-1,1.87)(.6,1.87)
\psline[linecolor=cyan,linewidth=1pt](1.87,-1)(1.87,.6)
\pscircle[linecolor=cyan,linewidth=1pt](.5,.5){1.25}
\psline[linecolor=red,linewidth=.5pt]{<->}(0,0)(1.45,1.45)
\rput[h](1.1,1){\red\tiny$2\sqrt{2}x^2+x^2+1$}
\pscircle[linecolor=blue,linewidth=1pt](.5,.5){1.25}
\rput[h](1,2){\blue\footnotesize$F(x)$}
\end{pspicture}}
\subfigure[$F(x)$ and $(2\sqrt{2}x^2+x^2+1)e+F(x)-C$]{\begin{pspicture}(1,1)(-3.4,-3.4)
\psaxes[arrows=->](0,0)(-3,-3)(1,1)
\rput[h](0.2,0.2){\black{{\footnotesize\textbf{O}}}}
\rput[h](1.2,0){\black{$x$}}
\rput[h](0,1.2){\black{$y$}}
\pscustom[linecolor=white,linewidth=0.001pt,fillstyle=hlines,ticks=none,hatchwidth=.5pt,hatchcolor=cyan]{
\pscircle[linecolor=blue,fillstyle=solid,hatchwidth=.5pt,fillcolor=white,hatchcolor=white](-1,-1){1.25}
\psarc(-1,-1){1.25}{0}{90}}
\psline[linecolor=cyan,linewidth=1pt](.35,-1.17)(.35,-2.9)
\psline[linecolor=cyan,linewidth=1pt](-1.17,.35)(-2.9,.35)
\pscircle[linecolor=cyan,linewidth=1pt](-1,-1){1.25}
\pscustom[linecolor=white,linewidth=0.001pt,fillstyle=hlines,ticks=none,hatchwidth=.5pt,hatchcolor=cyan]{
\psline(-2.9,.3)(.3,-1.17)
\psline(.3,-1.17)(.3,-2.9)
\psline(.3,-2.9)(-2.9,-2.9)
\psline(-2.9,-2.9)(-2.3,.3)}
\pscustom[linecolor=white,linewidth=0.001pt,fillstyle=hlines,ticks=none,hatchwidth=.5pt,hatchcolor=cyan]{
\psline(.3,-2.9)(-1.17,.3)
\psline(-1.17,.3)(-2.9,.3)
\psline(-2.9,.3)(-2.9,-2.9)
\psline(-2.9,-2.9)(.3,-2.3)}
\rput[h](-1.5,-3.2){\cyan\footnotesize$(2\sqrt{2}x^2+x^2+1)e+F(x)-C$}
\pscircle[linecolor=blue,linewidth=1pt](-1,-1){1.25}
\rput[h](-1,-2){\blue\footnotesize$F(x)$}
\end{pspicture}}
\end{center}
\caption{The calculation of $G_e^u(F(x),\{0\})$ geometrically}
\label{fig22}
\end{figure}
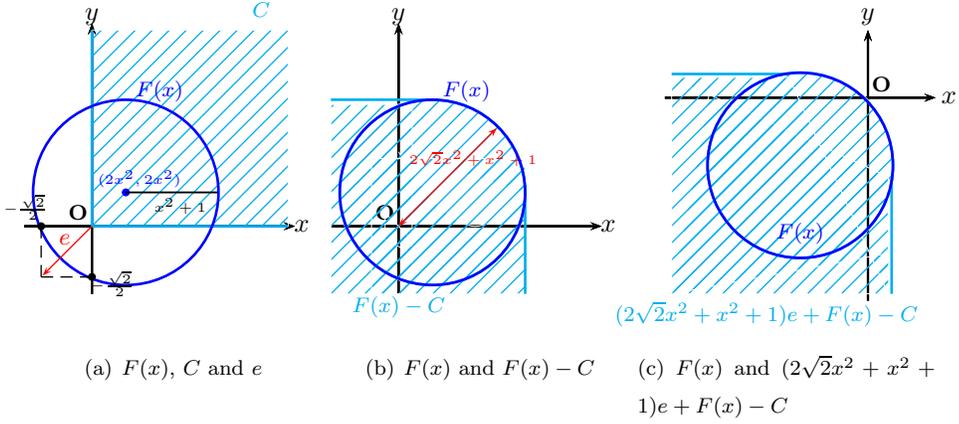

Finally, we have \begin{equation}
\label{eq2}
\begin{array}{ll}
  v_e(F(x)) & =\left(-G_e^{\ell}(\{0\},F(x)),G_e^u(F(x),\{0\})\right) \\
   & =\left(\sqrt{2}(x^2-1),2\sqrt{2}x^2+x^2+1\right).
\end{array}\end{equation}

\end{example}
Consequently, above calculations point out that by choosing a suitable $e\in -int(C)$ we obtained $w_e$ and $v_e$ easily. However, vectorizing function in [19, Example 3.1] was obtained by considering all vectors of the polar cone of $C$.

\newpage

\section*{Gerstewitz vectorization and optimality conditions for $(s-SOP)$}

 \indent $(s-SOP)$ can be replaced by a vector optimization problem using Gerstewitz vectorizing function. In this section, results of previous section are employed to give optimality conditions for $(s-SOP)$ without any convexity assumption and relationships between solutions of $(s-SOP)$ and $(VOP)$ derived by Gerstewitz vectorizing function.

\begin{theorem}
\label{u3}
  Let $F:X\rightrightarrows Y$ be $\mp C$-closed and $\mp C$-bounded valued on $X$. $x_0\in X$ is an $s$-maximal ($s$-minimal) solution of $(s-SOP)$ if and only if there exists an $s$-increasing ($s$-decreasing) function $T:\mathcal{P}^0_{\mp C}(Y)\rightarrow\R^2$ satisfying the following statements:
  \begin{itemize}
    \item[(i)] If $x\in X$ and $F(x)\in[F(x_0)]^s$, then $T(F(x))=(0,0)$,
    \item[(ii)] If $x\in X$ and $F(x)\not\in[F(x_0)]^s$, then $(0,0)\not\leq_{\R^2_+}T(F(x))$,
    \item[(iii)] If $A\in\mathcal{P}_{0\mp C}(Y)$ and $F(x_0)\preceq^sA$ ($A\preceq^sF(x_0)$), then $(0,0)\leq_{\R^2_+}T(A)$.
  \end{itemize}
\end{theorem}
\begin{proof}
  $(\Longrightarrow)$ Suppose that $x_0$ is an $s$-maximal solution of $(s-SOP)$. Let us fix any $e\in-int(C)$ and consider the function $T:\mathcal{P}^0_{\mp C}(Y)\rightarrow\R^2$ defined as \linebreak  $T(\cdot)=w_e(F(x_0),\cdot)=(-G_e^{\ell}(F(x_0),\cdot),G_e^u(\cdot,F(x_0)))$. By Theorem \ref{u1} (v) $T$ is \linebreak $s$-increasing on $\mathcal{P}^0_{\mp C}(Y)$. Now, we show that $T$ satisfies conditions (i)-(iii).
  \begin{itemize}
    \item[(i)] Since $F(x)\in[F(x_0)]^s$, we have $T(F(x))=w_e(F(x_0),F(x))=(0,0)$ from Theorem \ref{ut1} (ii).
    \item[(ii)] Let $F(x)\not\in[F(x_0)]^s$. Since $x_0$ is an $s$-maximal solution of $(s-SOP)$, we have $F(x_0)\not\preceq^sF(x)$. So, from Theorem \ref{ut1} (iii) we obtain $$(0,0)\not\leq_{\R^2_+}w_e(F(x_0),F(x))=T(F(x)).$$
    \item[(iii)] Assume that $F(x_0)\preceq^sA$. By Theorem \ref{ut1} (iii) we get $$(0,0)\leq_{\R^2_+}w_e(F(x_0),A)=T(A).$$
  \end{itemize}
$(\Longleftarrow)$ Let (i)-(iii) be satisfied for some $T:\mathcal{P}^0_{\mp C}(Y)\rightarrow\R^2$  which is $s$-increasing on $\mathcal{P}^0_{\mp C}(Y)$. Assume the contrary that $x_0$ isn't an $s$-maximal solution of  $(s-SOP)$. Then, there exists $x'\in X$ such that $F(x_0)\preceq^sF(x')$ and $F(x')\not\preceq^sF(x_0)$. Hence, $F(x')\not\in[F(x_0)]^s$. From (ii) we have \begin{equation}
\label{opt1}
(0,0)\not\leq_{\R^2_+}T(F(x')).\end{equation} Since $F(x_0)\preceq^sF(x')$, by (iii) we have $(0,0)\leq_{\R^2_+}T(F(x'))$. This contradicts (\ref{opt1}). Therefore, $x_0$ is an $s$-maximal solution of $(s-SOP)$.

It is enough to take the function $T(\cdot)=w_e(\cdot,F(x_0))$ to prove the minimality of $x_0$.
\end{proof}
\begin{theorem}
\label{u5}
  Let $F:X\rightrightarrows Y$ be $\mp C$-compact valued on $X$. $x_0\in X$ is a weak $s$-maximal (weak $s$-minimal) solution of  $(s-SOP)$ if and only if there exists an $s$-increasing (strictly $s$-decreasing) function $T:\mathcal{P}^0_{\mp C}(Y)\rightarrow\R^2$ satisfying the following statements:
  \begin{itemize}
    \item[(i)] If $x\in X$ and $F(x)\in[F(x_0)]^s$, then $T(F(x))=(0,0)$,
    \item[(ii)] If $x\in X$ and $F(x)\not\in[F(x_0)]^s$, then $(0,0)\not<_{\R^2_+}T(F(x))$,
    \item[(iii)] If $A\in\mathcal{P}^0_{\mp C}(Y)$ is a $\mp C$-compact set and $F(x_0)\prec^sA$ ($A\prec^sF(x_0)$), then $(0,0)<_{\R^2_+}T(A)$.
  \end{itemize}
\end{theorem}

\begin{theorem}
\label{u110}
Let $F:X\rightrightarrows Y$ be $\mp C$-closed, $\mp C$-bounded valued on $X$ and \linebreak $e\in -int(C)$. $x_0\in X$ is an $s$-maximal ($s$-minimal) solution of $(s-SOP)$  if and only if $x_0$ is a solution of the problem

   $
(VOP_w^s) \left\{ \begin{array}{ll}
\max w_e(F(x_0),F(x)) & \\
 s.t. \ x\in X &
\end{array} \right.
$ $\Bigg(
 \left\{ \begin{array}{ll}
\max w_e(F(x),F(x_0)) & \\
 s.t. \ x\in X &
\end{array} \right.
\Bigg).$

\end{theorem}
\begin{proof}
  $(\Longrightarrow)$ It is a result of Theorem \ref{u3}.

  $(\Longleftarrow)$ Let $x_0$ be a solution of $(VOP_w^s)$. Then, we have $(0,0)\not\leq_{\R^2_+}w_e(F(x_0),F(x))$ for all $x\in X$ where $F(x)\not\in[F(x_0)]^s$. By Theorem \ref{ut1} (iii) $F(x_0)\not\preceq^sF(x)$ for all $x\in X$ where $F(x)\not\in[F(x_0)]^s$. Thus, $x_0$ is a solution of $(s-SOP)$.
\end{proof}
\begin{theorem}
   Let $F:X\rightrightarrows Y$ be $\mp C$-compact valued on $X$ and $e\in -int(C)$. $x_0\in X$ is a weak $s$-maximal (weak $s$-minimal) solution of $(s-SOP)$ if and only if $x_0$ is a strongly solution of the problem
$$
(VOP_w^s) \left\{ \begin{array}{ll}
\max w_e(F(x_0),F(x)) & \\
 s.t. \ x\in X &
\end{array} \right.
\Bigg(
 \left\{ \begin{array}{ll}
\max w_e(F(x),F(x_0)) & \\
 s.t. \ x\in X &
\end{array} \right.
\Bigg).$$
\end{theorem}
\begin{proof}

 $(\Longrightarrow)$ It is a result of Theorem \ref{u5}.

 $(\Longleftarrow)$ It can be proved by using Theorem \ref{ut1} (iii).
\end{proof}

\begin{corollary}
\label{teo2}
 Let $F:X\rightrightarrows Y$ be $\mp C$-closed and $\mp C$-bounded valued on $X$ and
 \begin{equation}
  \label{son5}
    F(x)\preceq^sF(y) \text{ or } F(y)\preceq^sF(x) \ \text{ for all }x,y\in X.
  \end{equation} If $x_0\in X$ is an $s$-maximal ($s$-minimal) solution of $(s-SOP)$, then it is also a strongly solution of the problem: $$(VOP_v^s)\left\{
  \begin{array}{c}
  \max v_e(F(x)) \\
  s.t. \ x\in X \\
   \end{array}
   \right.\Bigg(\left\{
   \begin{array}{c}
   \min v_e(F(x)) \\
   s.t. \ x\in X \\
   \end{array}
    \right.\Bigg).$$
\end{corollary}
\begin{proof}
  Let $x_0$ be an $s$-maximal solution of $(s-SOP)$. If $F(x)\in[F(x_0)]^s$, then $F(x)\preceq^sF(x_0)$. If $F(x)\not\in[F(x_0)]^s$, then from $s$-maximality of $x_0$ and (\ref{son5}) we have $F(x)\preceq^sF(x_0)$. So, we obtain $F(x)\preceq^sF(x_0)$ for all $x\in X$. Since $v_e(\cdot)$ is $s$-increasing, we get $v_e(F(x))\leq_{\R^2_+}v_e(F(x_0))$ for all $x\in X$. Therefore, $x_0$ is a strongly solution of $(VOP_v^s)$.
\end{proof}
The following example shows that the condition (\ref{son5}) is necessary in Corollary \ref{teo2}.
\begin{example}
  Let $Y=\R^2$, $C=\R^2_+$, $X=\{1,2\}$, $A=[1,2]\times\{1\}$, $B=\{(\frac{3}{2},2)\}$, $F:X\rightrightarrows Y$ be defined as $F(1)=A$, $F(2)=B$. Consider the problem
  $$(s-SOP)\left\{
  \begin{array}{ll}
    \max F(x) & \\
    s.t. \ x\in \{1,2\}. &  \\
        \end{array}
        \right.$$

\begin{figure}[h]
\begin{center}
\begin{pspicture}(-2,-2)(3.1,3.1)
\psaxes[arrows=->](0,0)(-2,-2)(3,3)
\rput[h](-0.3,-0.3){\black{{\textbf{O}}}}
\rput[h](3.2,0){\black{\large$x$}}
\rput[h](0,3.2){\black{\large$y$}}
\psline[showpoints=false,linewidth=1pt,fillstyle=vlines,linecolor=blue](1,1)(2,1)
\psdot[linewidth=.5pt,linecolor=red](1.5,2)
\psdot[linewidth=.5pt,linecolor=blue](1,1)
\psdot[linewidth=.5pt,linecolor=blue](2,1)
\rput[h](1.5,1.2){\blue$A$}
\rput[h](1.5,2.3){\red$B$}
\end{pspicture}\end{center}
\caption{Image sets of $F$}\label{sekil145}
\end{figure}
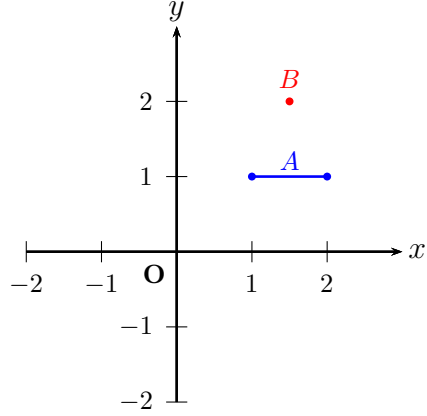
\end{example}
As seen in Fig. \ref{sekil145}, $A\not\preceq^sB$ and $B\not\preceq^sA$, i.e., (\ref{son5}) is not satisfied for this problem. Since $A\not\preceq^sB$ and $B\not\preceq^sA$, solutions of $(s-SOP)$ are 1 and 2.

Let us choose $e=(-1,-1)$. We have $$v_e(F(1))=(-G_e^{\ell}(\{0\},F(1)),G_e^u(F(1),\{0\}))=(1,1)$$ and
   $$v_e(F(2))=(-G_e^{\ell}(\{0\},F(2)),G_e^u(F(2),\{0\}))=\left(\frac{3}{2},\frac{3}{2} \right).$$
But, the unique solution of the problem $$(VOP_v^s)\left\{
\begin{array}{c}
\max v_e(F(x)) \\
 s.t. \ x\in X. \\
 \end{array}
 \right.$$
is $x_0=2$. $x_1=1$ is a solution of $(s-SOP)$, but it can not be obtained by Gerstewitz vectorization.
\begin{corollary}
   Let $F:X\rightrightarrows Y$ be $\mp C$-compact valued on $X$ and
   \begin{equation}
  \label{son55}
    F(x)\prec^sF(y) \text{ or } F(y)\prec^sF(x) \ \text{ for all }x,y\in X.
  \end{equation} If $x_0\in X$ is a weak  $s$-maximal (weak $s$-minimal) solution of $(s-SOP)$, then it is also a strongly solution of the problem $$(VOP_v^s)\left\{
  \begin{array}{c}
  \max v_e(F(x)) \\
  s.t \ x\in X \\
   \end{array}
   \right.\Bigg(\left\{
   \begin{array}{c}
   \min v_e(F(x)) \\
   s.t. \ x\in X \\
   \end{array}
    \right.\Bigg).$$
\end{corollary}
\begin{proof}
  It can be proved using strictly monotonicity of $v_e(\cdot)$.
\end{proof}
\begin{theorem}
\label{th1}
  Let $F:X\rightrightarrows Y$ be $\mp C$-closed, $\mp C$-bounded valued on $X$ and \linebreak $v_e(F(x))\neq v_e(F(y))$ for all $x,y\in X$, $x\neq y$. If $x_0\in X$ is a maximal (minimal) solution of $(VOP_v^s)$, then it is also an $s$-maximal ($s$-minimal) solution of $(s-SOP)$.
\end{theorem}
\begin{proof}
 Let $x_0$ be a maximal solution of $(VOP_v^s)$. Then, we have \linebreak $v_e(F(x_0))\not\leq_{\R^2_+}v_e(F(x))$ for all $x\in X$ such that $F(x)\not\in[F(x_0)]^s$. As $v_e(\cdot)$ is \linebreak $s$-increasing, we get   $G_e^{\ell}(\{0\},F(x))\not\leq G_e^{\ell}(\{0\},F(x_0))$ or $G_e^u(F(x_0),\{0\})\not\leq G_e^u(F(x),\{0\})$. Since $G_e^{\ell}(\{0\},\cdot)$ and $G_e^u(\cdot,\{0\})$ are $\ell$-decreasing, we have $F(x_0)\not\preceq^{\ell}F(x)$ or \linebreak $F(x_0)\not\preceq^uF(x)$ (from (\ref{eq4})), respectively. Hence, we obtain $F(x_0)\not\preceq^sF(x)$ for all $x\in X$ such that $F(x)\not\in[F(x_0)]^s$. Therefore, $x_0$ is an $s$-maximal solution of $(s-SOP)$.
\end{proof}

We construct Gerstewitz vectorization for $(s-SOP)$ given in [19, Example 3.1] with convex objective map in the following example.

\begin{example}
  Let $Y=\R^2$, $C=\R^2_+$ and $F:[-1,1]\rightrightarrows Y$ be defined as  $$F(x):=\left\{(y_1,y_2)\in\R^2 \ | \ (y_1-2x^2)^2+(y_2-2x^2)^2\leq(x^2+1)^2\right\}$$ for all $x\in[-1,1]$. Consider,

$$(s-SOP)\left\{ \begin{array}{ll}
\min F(x) & \\
 s.t. \ x\in [-1,1]. &
\end{array} \right.$$
 Because $F(x)=F(-x)$ for all $x\in[-1,1]$, we consider the problem

$$(s-SOP)\left\{ \begin{array}{ll}
\min F(x) & \\
 s.t. \ x\in [0,1]. &
\end{array} \right.$$

Let us choose $e=\left(-\frac{\sqrt{2}}{2},-\frac{\sqrt{2}}{2}\right)$ and consider the problem

 $$
(VOP_v^s)
 \left\{ \begin{array}{ll}
\min v_e(F(x)) & \\
 s.t. \ x\in [0,1]. &
\end{array} \right.$$

 As seen in Fig. \ref{fig18} and Fig. \ref{fig15}, $F:X\rightrightarrows Y$ is $\mp C$-closed, $\mp C$-bounded valued on $X$ and $v_e(F(x))\neq v_e(F(y))$ for all $x\neq y$ and $x,y\in [0,1]$. From (\ref{eq2}) we have $$v_e(F(x))=\left(\sqrt{2}(x^2-1),2\sqrt{2}x^2+x^2+1\right).$$

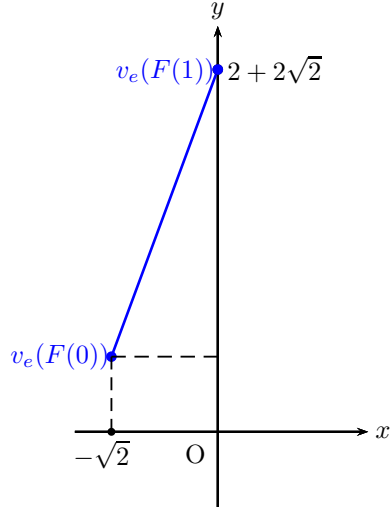
\begin{figure}[h]
\begin{center}
\begin{pspicture}(-2,-1)(2,5.5)
\psaxes[Dx=12,Dy=11,arrows=->](0,0)(-1.9,-1)(2,5.4)
\rput[h](-0.3,-0.3){\black{{O}}}
\psline[linewidth=1pt,fillstyle=vlines,linecolor=blue](-1.414,1)(0,4.82)
\rput[h](2.2,0){\black{$x$}}
\rput[h](0,5.6){\black{$y$}}
\rput[h](-2.1,1){\blue$v_e(F(0))$}
\rput[h](-.7,4.8){\blue$v_e(F(1))$}
\rput[h](.75,4.82){$2+2\sqrt{2}$}
\psdot[linewidth=.5pt](-1.414,0)
\psdot[linewidth=1pt,linecolor=blue](-1.414,1)
\rput[h](-1.55,-.3){$-\sqrt{2}$}
\psdot[linewidth=1pt](0,4.82)
\psline[linestyle=dashed,linewidth=.7pt](-1.414,0)(-1.414,1)
\psline[linestyle=dashed,linewidth=.7pt](-1.414,1)(0,1)
\psdot[linewidth=1pt,linecolor=blue](0,4.82)
\end{pspicture}\end{center}
\caption{Image set of $v_e(F(x))$}\label{fig15}
\end{figure}

Also, there isn't any $x\in(0,1]$ such that $$v_e(F(x))\leq_{\R^2_+} v_e(F(0)).$$ So, $x_0=0$ is the minimal solution of $(VOP_v^s)$. Therefore, $x_0=0$ is a solution of $(s-SOP)$ by Theorem \ref{th1}.
\end{example}

Now, we construct Gerstewitz vectorization for a nonconvex $(s-SOP)$.

\begin{example}
  Let $Y=\R^2$, $C=\R^2_+$ and $F:[0,2]\rightrightarrows Y$ be defined as  $$F(x):=\left\{
                                     \begin{array}{ll}
                                       \left([x-2,x]\times[x-2,x]\right)\cup\{(6+x,6+x)\} & ;x\in[0,2) \\
                                       \mathrm{conv\{(-5,0),(6,6)\}\cup conv\{(0,-5),(6,6)\}} & ;x=2. \\
                                     \end{array}
                                   \right.$$

Consider the problem
$$(s-SOP)\left\{ \begin{array}{ll}
\min F(x) & \\
 s.t. \ x\in [0,2]. &
\end{array} \right.$$

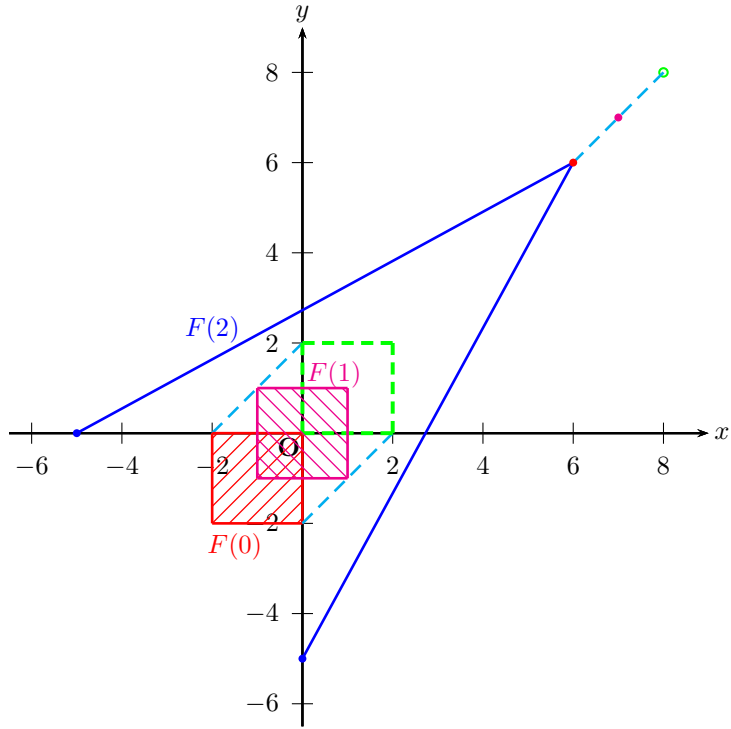
\begin{figure}[h]
\begin{center}
\psset{Dy=2,Dx=2,xunit=6mm,yunit=6mm}\begin{pspicture}(-6.5,-6.5)(9,9)
\psaxes[arrows=->](0,0)(-6.5,-6.5)(9,9)
\rput[h](-0.3,-0.3){\black{{\textbf{O}}}}
\rput[h](9.3,0){\black{$x$}}
\rput[h](0,9.3){\black{$y$}}
\psline[showpoints=false,linewidth=1pt,fillstyle=vlines,linecolor=blue](-5,0)(6,6)
\psline[showpoints=false,linewidth=1pt,fillstyle=vlines,linecolor=blue](0,-5)(6,6)
\rput[h](-2,2.3){\blue$F(2)$}
\pscustom[linecolor=white,linewidth=0.001pt,fillstyle=hlines,ticks=none,hatchwidth=.5pt,hatchcolor=red]{
\psline(-2,-2)(0,-2)
\psline(0,-2)(0,0)
\psline(0,0)(-2,0)
\psline(-2,0)(-2,-2)}
\psline[showpoints=false,linewidth=1pt,fillstyle=vlines,linecolor=red](-2,-2)(0,-2)
\psline[showpoints=false,linewidth=1pt,fillstyle=vlines,linecolor=red](0,-2)(0,0)
\psline[linestyle=dashed,linewidth=1pt,linecolor=cyan](-2,0)(0,2)
\psline[linestyle=dashed,linewidth=1pt,linecolor=cyan](0,-2)(2,0)
\psline[linestyle=dashed,linewidth=1.5pt,linecolor=green](2,0)(2,2)
\psline[linestyle=dashed,linewidth=1.5pt,linecolor=green](0,2)(2,2)
\psline[linestyle=dashed,linewidth=1.5pt,linecolor=green](2,0)(0,0)
\psline[linestyle=dashed,linewidth=1.5pt,linecolor=green](0,2)(0,0)
\psline[showpoints=false,linewidth=1pt,fillstyle=vlines,linecolor=red](0,0)(-2,0)
\psline[showpoints=false,linewidth=1pt,fillstyle=vlines,linecolor=red](-2,0)(-2,-2)
\rput[h](-1.5,-2.5){\red$F(0)$}
\pscustom[linecolor=white,linewidth=0.001pt,fillstyle=vlines,ticks=none,hatchwidth=.5pt,hatchcolor=magenta]{
\psline(-1,-1)(1,-1)
\psline(1,-1)(1,1)
\psline(1,1)(-1,1)
\psline(-1,1)(-1,-1)}
\psline[showpoints=false,linewidth=1pt,fillstyle=vlines,linecolor=magenta](-1,-1)(1,-1)
\pscircle[linecolor=green,fillstyle=solid,hatchwidth=.5pt,fillcolor=white,hatchcolor=white](8,8){.07}
\psline[linestyle=dashed,linewidth=1pt,linecolor=cyan](6,6)(8,8)
\psline[showpoints=false,linewidth=1pt,fillstyle=vlines,linecolor=magenta](1,-1)(1,1)
\psline[showpoints=false,linewidth=1pt,fillstyle=vlines,linecolor=magenta](1,1)(-1,1)
\psline[showpoints=false,linewidth=1pt,fillstyle=vlines,linecolor=magenta](-1,1)(-1,-1)
\rput[h](.7,1.3){\magenta$F(1)$}
\psdot[linewidth=.5pt,linecolor=blue](-5,0)
\psdot[linewidth=.5pt,linecolor=blue](0,-5)
\psdot[linewidth=.5pt,linecolor=magenta](7,7)
\psdot[linewidth=.5pt,linecolor=red](6,6)
\end{pspicture}\end{center}
\caption{Some image sets of $F$}
\label{fig}
\end{figure}

 Some image sets of $F$ are given in Fig. \ref{fig}. Since $F(x)\not\preceq^{\ell}F(2)$, we have \linebreak $F(x)\not\preceq^sF(2)$ for all $x\in[0,2)$. Then, $x_0=2$ is an $s$-minimal solution of $(s-SOP)$. As $F(x)\not\preceq^{\ell}F(0)$, we get $F(x)\not\preceq^sF(0)$ for all $x\in(0,2]$. Hence, $x_0=0$ is an \linebreak $s$-minimal solution of $(s-SOP)$. Let us choose $x\in(0,2)$. We get $F(0)\preceq^sF(x)$ and $F(x)\not\preceq^sF(0)$. So, $x$ isn't an $s$-minimal solution of $(s-SOP)$. Therefore, solutions of $(s-SOP)$ are 0 and 2.

Since $F(2)+C$ isn't convex, vectorization in \cite{Jahn2} couldn't be applied to this problem. But, we can solve this problem via Gerstewitz vectorization.

Now, we show that $x_0=2$ is a solution of this problem by using Theorem \ref{u110}.

Let us choose $e=(-1,-1)$ and consider the problem

$$
(VOP_w^s)
 \left\{ \begin{array}{ll}
\max w_e(F(x),F(2)) & \\
 s.t. \ x\in [0,2]. &
\end{array} \right.$$

 We get
$$\begin{array}{ll}
   w_e(F(x),F(2)) & =(-G_e^{\ell}(F(x),F(2)),G_e^u(F(2),F(x))) \\
   \ \\
    &  =\left\{\begin{array}{cc}
    (-3-x,-x) & ;x\neq2 \\
    (0,0) & ;x=2.
    \end{array}\right.\end{array}$$

     \begin{figure}[h]
\begin{center}
\begin{pspicture}(-6,-3)(2.1,2.1)
\psaxes[Dx=3,Dy=4,arrows=->](0,0)(-5.98,-3)(2,2)
\rput[h](-0.3,-0.3){\black{{O}}}
\psline[showpoints=false,linewidth=1pt,fillstyle=vlines,linecolor=blue](-3,0)(-5,-2)
\rput[h](2.2,0){\black{$x$}}
\rput[h](0,2.2){\black{$y$}}
\rput[h](-3,.5){\blue$w_e(F(0),F(2))$}
\rput[h](1.5,.5){\red$w_e(F(2),F(2))$}
\rput[h](.4,-2){-2}
\rput[h](-5.2,-.4){-5}
\psdot[linewidth=1pt,linecolor=red](0,0)
\psdot[linewidth=1pt](-5,0)
\psdot[linewidth=1pt](0,-2)
\psline[linestyle=dotted](-5,-2)(-5,0)
\psline[linestyle=dotted](-5,-2)(0,-2)
\psdot[linewidth=1pt,linecolor=blue](-3,0)
\pscircle[linecolor=blue,fillstyle=solid,hatchwidth=.5pt,fillcolor=white,hatchcolor=white](-5,-2){.07}
\end{pspicture}\end{center}
\caption{Image set of $w_e(F(x),F(2))$}\label{fig1}
\end{figure}
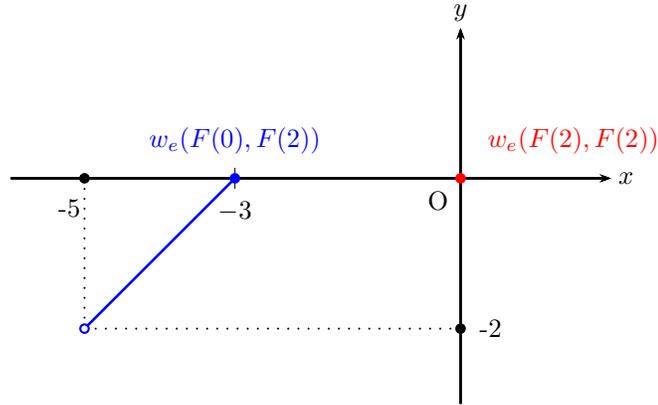

As seen in Fig. \ref{fig1} there isn't any $x\in[0,2)$ such that $$w_e(F(2),F(2))\leq_{\R^2_+} w_e(F(x),F(2)).$$ So, $x_0=2$ is the solution of $(VOP_w^s)$. Therefore, $x_0=2$ is a solution of $(s-SOP)$ by Theorem \ref{u110}.

$x_0=0$ is also a solution of $(s-SOP)$. It can be shown similarly via Gerstewitz vectorization.

\end{example}

\newpage
\section*{Conclusion}
In this study, our aim is to replace a nonconvex set-valued optimization problem with respect to set less order relation with a vector optimization problem via Gerstewitz vectorizing function. This can provide us to use known solution techniques such as scalarization, duality, derivative etc. in vector optimization to solve nonconvex set-valued optimization problems. For further studies, one can investigate the usage of these techniques in set-valued optimization via different vectorizations.


\end{document}